\documentclass[amstex,12pt]{article}

\usepackage{graphicx}
\usepackage{amsfonts}
\usepackage{amssymb}
\usepackage{amsmath}

\def\bs{{\bigskip}}

\def\e{{\varepsilon}}

\newtheorem{thm}{Theorem}
\newtheorem{lm}{Lemma}

\newtheorem{co}{Corollary}

\newenvironment{proof}[1][Proof]{\noindent\textsl{#1:} }{\hfill $\Box$}

\begin{document}

\author{Jin-ichi Itoh, Costin V\^\i lcu and Tudor Zamfirescu}
\title{With respect to whom are you critical?}
\maketitle
\date{}

\abstract{
For any compact Riemannian surface $S$ and any point $y$ in $S$,
$Q_y^{-1}$ denotes the set of all points in $S$, for which $y$ is a critical point.
We proved \cite{BIVZ} together with Imre B\'ar\'any that card$Q_y^{-1} \geq 1$, 
and that equality for all $y\in S$ characterizes the surfaces homeomorphic to the sphere.
Here we show, for any orientable surface $S$ and any point $y \in S$, the following two main results.
There exist an open and dense set of Riemannian metrics $g$ on $S$ for which $y$ is critical with respect to an odd number of points in $S$,
and this is sharp.
Card$Q_y^{-1} \leq 5$ for the torus and card$Q_y^{-1} \leq 8g-5$ if the genus $g$ of $S$ is at least $2$.
Properties involving points at globally maximal distance on $S$ are eventually presented.
}

\bigskip

{\small {\bf Math. Subj. Classification (2000):} 53C45}


\section{Introduction}

In this paper, by {\it surface} we always mean a $2$-dimensional compact Riemannian manifold, unless explicitly stated otherwise.

For any surface $S$, denote by $\rho$ its (intrinsic) metric, and by $\rho _x$ the {\it distance function} from $x \in S$, given by $\rho _x(y)=\rho (x,y)$.
A {\it segment} between $x$ and $y$ in $S$ is a path from $x$ to $y$ of length $\rho (x,y)$.
A point $y \in S$ is called {\it critical} with respect to $\rho _x$ (or to $x$), if for any tangent direction $\tau$ of $S$ at $y$ there exists a segment from $y$ to $x$ whose tangent direction at $y$ makes a non-obtuse angle with $\tau$.

For an excellent survey of critical point theory for distance functions see \cite{gro}.

For any point $x$ in $S$, denote by $Q_x$ the set of all critical points 
with respect to $x$, and by $Q$ the {\it critical point mapping} 
associating to any point $x$ in $S$ the set $Q_x$.
Similarly, $M_x$ is the set of all relative maxima of $\rho_x$, $F_x$ the 
set of all farthest points from $x$ (i.e., absolute maxima of $\rho_x$)
 and $M$, respectively $F$, are the corresponding set-valued mappings.

Properties of the mappings $Q$, $M$ and $F$ on Alexandrov spaces
have previously been obtained in \cite{gp} and \cite{vz2}.
See the survey \cite{v2} for various results concerning the mapping $F$ on convex surfaces.

We proved in \cite{BIVZ}, together with Imre B\'ar\'any, that 
the set $Q_y^{-1}$ of all points with respect to which $y$ is critical is never empty. 
It is also shown in \cite{BIVZ} that $Q_y^{-1}$ is single-valued for all $y\in S$ if and only if the genus of $S$ is $0$.
We continue this study in the following.

Let ${\cal G}$ denote the space of all Riemannian metrics on the surface $S$;
it is viewed as the space of sections of the bundle of positive definite symmetric matrices over $S$, 
endowed with the ${\cal C}^\infty$ Whitney topology \cite{Bu}.

In a topological space $\cal T$, a property P is called generic if the set of all elements in $\cal T$ without property P is of first Baire category. 
We obtain an even stronger sense of genericity if ``nowhere density'' replaces ``first Baire category'', and this is the meaning we use in this paper.
Several results and open questions about generic Riemannian metrics are presented in \cite{Berger},
see also the references therein. We mention next only one.

M. A. Buchner \cite{Bu} showed that, on a surface, the set of metrics which are cut locus stable is open and dense in ${\cal G}$; 
moreover, for any such metric, every ramification point of the cut locus has degree three. 
We get, and later use, a slightly improved result, 
see $\S$\ref{prelim} for the definitions and Theorem \ref{generic_Cx} for the precise statement.

Our Theorem \ref{odd_1} contributes to this topic, too. It states that 
any point $y$ in any orientable surface $S$ is critical with respect to an odd number of points in $S$, 
for a generic metric on $S$. This result is sharp, as Theorem \ref{even} shows.
Theorem \ref{odd_1} is also useful for the proof of our Theorem \ref{g}.

Theorem \ref{g} provides, for orientable Riemannian surfaces, an upper bound for card$Q_y^{-1}$.
It is based on its counter-part for Alexandrov surfaces, Theorem \ref{con}, which strengthens Theorem 2 in \cite{Z2}.
We apply Theorem \ref{g} to estimate the cardinality of diametrally opposite sets on $S$ (Corollary \ref{diam}). 
Thus, our results also contribute to a description of farthest points H. Steinhaus had asked for (see $\S$A35 in \cite{cfg}).

The case of points $y$ in orientable Alexandrov surfaces, 
which are common maxima of several distance functions, is treated in \cite{v3};
for an introduction to Alexandrov spaces with curvature bounded below, see \cite{BGP}.
See also \cite{ro2}, \cite{ro4}, for results in a direction somewhat similar to ours.


\section{Preliminaries}
\label{prelim}

The length (1-dimensional Hausdorff measure) of the set $A$ is
denoted by $\lambda (A)$.

Let $S$ be a surface.
By $T_x$ we denote the circle of all tangent directions at $x\in S$; we have $\lambda (T_x) =2\pi$.

Let $x \in S$.
For every $\tau \in T_x$, a point $c(\tau)$ called {\it cut point} is associated, 
defined by the requirement that the arc $xc(\tau) \subset \Gamma$ is a segment which cannot be extended further (as a segment) beyond $c(\tau)$; here, $\Gamma$ is the geodesic through $x$ of tangent direction $\tau$ at $x$.
The set of all these cut points is the {\it cut locus} $C(x)$ of the point $x$.
The cut locus was introduced by H. Poincar\'e in 1905 \cite{p} and became, since then, an important tool in Global Riemannian Geometry,
see for example \cite{ko}, \cite{sa}, or \cite{ST}.

It is known that $C(x)$, if it is not a single point, is a {\it local tree}
(i.e., each of its points $z$ has a neighbourhood $V$ in $S$ such that
 the component $K_z(V)$ of $z$ in $C(x)\cap V$ is a tree), even a tree if
 $S$ is homeomorphic to the sphere.
If $S$ is not a topological sphere, the {\it cyclic part of}
$C(x)$ is the minimal (with respect to inclusion) subset $C^{cp}(x)$ of
 $C(x)$,  whose removal from $S$ produces a topological (open) disk.
It is easily seen that $C^{cp}(x)$ is a local tree with finitely many 
ramification points and no extremities (see \cite{i-z}).

Recall that a {\it tree} is a set $T$ any two points of which can be
 joined by a unique Jordan arc included in $T$.
The {\it degree} of a point $y$ of a local tree is the number of components
of $K_y(V)\setminus \{y\}$ if the neighbourhood $V$ of $y$ is chosen
 such that $K_y(V)$ is a tree.
A point $y$ of the local tree $T$ is called an {\it extremity} of $T$
if it has degree 1, and a {\it ramification point} of $T$ if it has
 degree at least 3. A local tree is {\it finite} if it has finitely many points
 of degree different from 2.
An {\it internal edge} of the finite tree $T$ is a  Jordan arc in $T$
in which the endpoints and no other points  are ramification
points of $T$.

All these notions admit obvious extensions to Alexandrov surfaces.
Theorem 4 in \cite{ZP} and Theorem 1 in \cite{ZI} yield the existence
of Alexandrov surfaces $S$ on which the set of all extremities of any cut locus
is residual in $S$.

It is, however, known that $C(x)$ has an at most countable set $C_3(x)$ of ramification points. 
Let $C_3^{cp}(x)$ be the set of points of degree at least 3 in the finite local tree $C^{cp}(x)$.
We stress that the degree is not taken in $C(x)$, but in $C^{cp}(x)$.
It is known that $C_3^{cp}(x)$ is a finite set.

Let $S$ be a surface and $x\in Q_y^{-1}$; put $i(x)=2$ if
there are precisely 2 segments from $y$ to $x$, and $i(x)=3$ if there
are at least 3 segments from $y$ to $x$.
For $j=2,3$, we say that the {\it point} $x$ is {\it of type} $j$ if $i(x)=j$.
Put $\sharp_y^j= {\rm card}\{x\in Q_y^{-1} : i(x)=j\}$; clearly, card$Q_y^{-1}=\sharp_y^2 + \sharp_y^3.$

\medskip

In \cite{BIVZ} the authors proved together with Imre B\'ar\'any, in the framework of Alexandrov surfaces, 
the following three results.
(See \cite{gg} for a variational proof of the first one, valid for finite dimensional Riemannian manifolds.)

\begin{lm}
\label{1}
Every point on every surface is critical with respect to some point of the surface.
\end{lm}

\begin{lm}
\label{sep}
Assume $S$ is a Riemannian surface, $y$ a point in $S$, and 
$x \in Q_y^{-1}$ is such that the union $U$ of two segments from $x$ to $y$ disconnects $S$. 
If a component $S'$ of $S \setminus U$ meets no segment from $x$ to $y$ then $Q_y^{-1} \cap S'= \emptyset$. 
In particular, if the union of any two segments from $x$ to $y$ disconnects $S$ then $Q_y^{-1}=\{ x \}$.
\end{lm}

Lemma \ref{sep} shows, in particular, that on many surfaces there are points 
which are critical with respect to precisely one other point.

\begin{lm}
\label{0}
An orientable surface $S$ is homeomorphic to the sphere ${\rm S}^2$ if and only if 
each point in $S$ is critical with respect to precisely one other point of $S$.
\end{lm}


\section{A general result}

We prove in this section a result for arbitrary Alexandrov surfaces, which in particular holds for (Riemannian) surfaces.
Before giving it, we recall a result in graph theory.
All graphs we consider in the following are finite, connected, and may have loops and multiple edges.

\begin{lm}
\label{mnq}
Let $G$ be a connected graph with $m$ edges, $n$ vertices and $q$ generating cycles.
Then 
\\i) $m-n+1=q$;
\\ii) $m \leq 3(q-1)$ and $n \leq 2(q-1)$, with equality if and only if $G$ is cubic.
\end{lm}

\begin{proof}
The equality (i) is well known. For the inequalities (ii), fix $q$. It follows from the first part that $m$ and $n$ are maximal if and only if $G$ is cubic. 
In this case we have $3n=2m$ and we obtain $n=2(q-1)$, $m=3(q-1)$.
\end{proof}

\bigskip

Recall that a point $y$ in an Alexandrov surface is called {\it smooth} if
$\lambda (T_y) =2 \pi$, where $T_y$ is the space of tangent directions at $y$ (as defined, for example, in \cite{BGP}).

For the simplicity of our exposition, we see every graph $G$ as an $1$-dimensional simplicial complex.

\begin{thm}
\label{con}
Let $y$ be a smooth point on an orientable Alexandrov surface $S$ of genus $g$.

If $g=0$ then ${\rm card} Q_y^{-1} =1$.

If $g\geq 1$ then $\sharp_y^2 \leq 6g-3$ and $\sharp_y^3 \leq 4g-2$; this yields ${\rm card} Q_y^{-1} \leq 10g-5$.
\end{thm}

For any point $y$ on the standard projective plane, $Q^{-1}_y = Q_y$ is a circle, so one cannot drop the orientability condition in Proposition \ref{con}.

The restriction to smooth points in Theorem \ref{con} is essential, too.
Indeed, for any surface $S$ with a conical point $y$, if $\lambda (T_y)\le\pi$ then $Q_y^{-1}=S\setminus\{y\}$.
See \cite{v3} for properties of the sets $M_y^{-1}$ and $Q_y^{-1}$ in case $\pi \leq \lambda (T_y) < 2\pi$.

\bigskip

\begin{proof}
The case $g=0$ is covered by Lemma \ref{0}, so we may assume $g\geq 1$.
And in the virtue of Lemma \ref{sep}, we may consider only points $y \in S$ with $Q_y^{-1} \subset C^{cp}(y)$.

Assume for simplicity of the exposition that $C(y)= C^{cp}(y)$.

Consider $C^{cp}(y)=(V,E)$ as a graph, 
with $V=C_3^{cp}(y)$ and $E$ the set of components of $C^{cp}(y) \setminus V$. 
Call the elements of $V$ {\it vertices}, and the elements of $E$ {\it edges}.

We claim that the interior of each edge $I$ of $C^{cp}(y)$ contains at most one point $x \in Q_y^{-1}$.
To see this, assume there exists some point $x \in Q_y^{-1}$ interior to $I$.
Then there are two segments from $x$ to $y$, making at $y$ the angle $\pi$.
Since $y$ is smooth, $\lambda (T_y) = 2\pi$ and therefore the two images $x', x''$ of $x$ on $T_y$ are diametrally opposite.
Let $x_* \neq x$ be another point in the interior of  $I$, with images $x'_*, x''_*$ on $T_y$.
Since $S$ is orientable, the order on $T_y$ is either $x', x'_*, x''_*, x''$ or $x', x''_*, x'_*, x''$. In both cases $x'_*, x''_*$ cannot be diametrally opposite, hence $x_* \not\in Q_y^{-1}$.

Then, since $C^{cp}(y)$ has $2g$ generating cycles, Lemma \ref{mnq} gives $\sharp_y^2 \leq 6g-3$ and $\sharp_y^3 \leq 4g-2$, which together imply
${\rm card} Q_y^{-1} = \sharp_y^2 + \sharp_y^3 \leq 10g-5$.
\end{proof}

\bigskip

Notice that this upper bound on card $Q_y^{-1}$ is imposed only by the topology of $S$. 
We shall refine it in Section \ref{upb} by local geometrical considerations.


\section{Two generic results}

For the proof of Theorem \ref{odd_1}, we shall make use of the main result in \cite{Bu}, 
that we complete in the following with a new statement of independent interest. 
Notice that this result doesn't require orientability of the surface.
See \cite{Bu} for the definition of cut locus stable metrics.

\begin{thm}
\label{generic_Cx}
Let $S$ be a surface and $y$ a point in $S$.

The set ${\cal C}^y$ of cut locus (with respect with $y$) stable metrics on $S$ is open and dense in ${\cal G}$. 
For any ${\mathfrak g}$ in ${\cal C}^y$, every ramification point of the cut locus $C(y)$ with respect to ${\mathfrak g}$ is joined to $y$ by precisely three segments.

There exists a set ${\tilde {\cal C}}^y$ of cut locus (with respect with $y$) stable metrics on $S$, open and dense in ${\cal G}$, such that 
for any ${\tilde{\mathfrak g}}$ in ${\tilde {\cal C}}^y$, 
every ramification point $x$ of the cyclic part of the cut locus $C(y)$ with respect to ${\mathfrak g}$ is joined to $y$ by precisely three segments, 
no two of them of opposite tangent directions at $x$ or at $y$.
\end{thm}

\begin{proof}
The first part of the theorem is proved by Buchner in \cite{Bu}.

Consider now a metric ${\mathfrak g}$ in ${\cal C}^y$ and a point $x \in C_3^{cp}(y)$, hence it is joined to $y$ by precisely three segments,
say $\gamma^1, \gamma^2, \gamma^3$.
Assume that the tangent directions of $\gamma^1$ and $\gamma^2$ at $x$ are opposite, so they form a {\it geodesic loop}.
Since the limit of geodesic loops is a geodesic loop, the set of all such metrics is closed in ${\cal G}$, and its complement ${\tilde {\cal C}}_x^y$ is open.

We prove now the density of ${\tilde {\cal C}}_x^y$ in ${\cal G}$.
In order to do it, we approximate ${\mathfrak g}$ in two steps. 

First we ``put a bump'' to slightly cover $\gamma^1$, assymetrically with respect to the left and right parts of $\gamma^1$.
Consequently, in a neighbourhood of the image set of $\gamma^1$ (on $S$), 
there is a unique shortest path $\tilde{\gamma}^1$ from $x$ to $y$ with respect to the new metric ${\mathfrak g}'$;
$\tilde{\gamma}^1$ is a little longer than $\gamma^1$ and, more importantly, it makes no angle of $\pi$ with $\gamma^2$ or $\gamma^3$.

Second, we put bumps on $\gamma^2$ and $\gamma^3$, such that the obtained metric ${\mathfrak g}''$ has the following properties.
In respective neighbourhoods of the image sets of $\gamma^2$ and $\gamma^3$ (on $S$), 
there are unique shorthest paths $\tilde{\gamma}^2$, $\tilde{\gamma}^3$ from $x$ to $y$, with respect to ${\mathfrak g}''$. 
They have the same respective tangent directions at $x$ as $\gamma^2$ and $\gamma^3$ and, moreover, they have the same length as $\tilde{\gamma}^1$.

One can proceed similarly for the tangent directions at $y$.

Since $C_3^{cp}(x)$ is a finite set, after finitely many such procedures we get a metric ${\tilde{\mathfrak g}} \in {\tilde {\cal C}}^y$
approximating ${\mathfrak g}$, with the desired properties.
\end{proof}

\begin{thm}
\label{odd_1}
If $S$ is an orientable surface and $y$ a point in $S$ then, for a generic Riemannian metric on $S$, 
$y$ is critical with respect to an odd number of points in $S$.
\end{thm}

During our proof we shall refer to the proof of Theorem 1 in \cite{BIVZ}.

\bigskip

\begin{proof}
Consider a metric on $S$ as in Theorem \ref{generic_Cx}.
We will identify here $T_y$ with a Euclidean circle of centre {\bf 0} and length $\lambda (T_y) = 2\pi.$

\bs

If $S$ is homeomorphic to the sphere then the statement follows from Lemma \ref{0}. Assume this is not the case.

A finite number of cycles were defined in \cite{BIVZ} to prove Lemma \ref{1}, 
by joining points in $T_y$ corresponding to the vertices in $C_3^{cp}(y)$ by line segments or arcs in $T_y$.
Next we indicate a geometrical interpretation (i.e., an equivalent definition) for some of those cycles, useful for our purpose.

The injectivity radius inj$(S)$ is positive. Therefore, for any $\varepsilon >0$ sufficiently small, there is a natural identification $\Phi$ 
of $T_y$ to the boundary ${\rm bd} N_\varepsilon$ of the $\varepsilon$-neighbourhood $N_\varepsilon$ of $C(y)$ in $S$.
Choose a point $x \in C_3^{cp}(y)$. For each segment $\gamma_x$ from $x$ to $y$, take the (first) point $z_{\gamma_x}$ in $\gamma_x \cap {\rm bd} N_\varepsilon$.
The set of all these points $z_{\gamma_x}$ has ${\rm deg} x = {\rm card} c^{-1}(x)$ components, each of which is a point or an arc.
(Recall that $c$ is the restriction of the exponential map to $T_y$.)
Join with segments the extremities of consecutive -- with respect to some circular order -- components.
The simple closed curve $C_x$ thus constructed corresponds, by the use of $\Phi^{-1}$, to the cycle $C_i$ determined by $c^{-1}(x)$, and is called a {\it vertex-cycle}. 
Moreover, the boundary of every component of $N_\varepsilon \setminus \bigcup_{x \in C_3^{cp}(x)} {\rm int} C_x$ yields, 
again by the use of $\Phi^{-1}$, a cycle $C_i$ determined by consecutive points $\alpha, \beta$ in $c^{-1} \left( C_3^{cp}(y) \right)$, and is called an {\it edge-cycle}.

Let $C_1,..., C_n$ be all these cycles.

If ${\bf 0}\in\cup_{j=1}^nC_j$ then, for some $x\in C_3^{cp}(y)$, there are two segments of diametrally opposite tangent directions in $T_y$, 
see \cite{BIVZ}.

If ${\bf 0}\not\in\cup_{j=1}^nC_j$ consider, as in \cite{BIVZ}, 
the winding number $w(C_j)=w({\bf 0}, C_j)$ of every cycle $C_j$ with respect to ${\bf 0}$. We have
$$\sum_{i=1}^n w \left( C_i \right) = w \left( \sum_{i=1}^n C_i \right) = w \left( T_y \right) =1
\hspace{5mm} ({\rm mod} \; 2),$$
because each edge not in $T_y$ is used exactly twice. This shows that $w(C_i) \neq 0$ for some cycle $C_i$.

If this cycle $C_i$ is an edge-cycle then, because $S$ is orientable, a semi-continuity argument shows that its corresponding edge in $C_3^{cp}(y)$ contains at least one point in $Q^{-1}_y$, see \cite{BIVZ} for details.

If $C_i$ is a vertex-cycle, $w(C_i) \neq 0$ means that {\bf 0} is surrounded by $C_i$, which is impossible if ${\bf 0}\notin {\rm conv} C_i$. By construction, conv$C_i=$ conv$c^{-1}(x)$ for some $x\in C_3^{cp}(y)$.

\bigskip

By Theorem \ref{generic_Cx}, the vertices of $C(y)$ all have degree three. This and the orientability of $S$ show now that all cycles $C_i$ considered above are simple closed curves, hence $w(C_i) \in \{0, \pm1\}$. Therefore, because 
$\sum_{i=1}^n w(C_i) = 1 \; ({\rm mod} \; 2)$, the number of cycles $C_i$ with $w(C_i) \neq 0$ is odd.
Each such $C_i$ intersects $Q_y^{-1}$.

We claim that, if non-zero, card$\left( C_i \cap Q_y^{-1} \right)=1$.
This is clear for the cycles determined by vertices of $C_3^{cp}(y)$, because these cycles have precisely three sides.
Consider now a cycle $C_i$ determined by an edge $e$ of $C_3^{cp}(y)$. Then, because  $S$ is orientable, 
$C_i$ has the form $\alpha_+\beta_-\beta^-\alpha^+$, with $\alpha_+\beta_-$ and $\alpha^+\beta^-$ of contrary orientations on $T_y$.
Take $x \in C_i \cap Q_y^{-1} \neq \emptyset$ and define
$l(x)=c^{-1}(x) \cap \alpha_+\beta_-$ and $r(x)=c^{-1}(x) \cap \alpha^+\beta^-$.
Of course, $l(x)$ and $r(x)$ contain each a single tangent direction for $x \in C(y) \setminus C_3(y)$, 
and at least one of them has at least two tangent directions for $x \in C_3(y)$.
In any case, let $l_{\alpha}(x)$ be the tangent direction in $l(x)$ closest to $\alpha_+$ along the arc $\alpha_+\beta_-$, 
and let let $l_{\beta}(x)$ be the tangent direction in $l(x)$ closest to $\beta_+$ along the same arc $\alpha_+\beta_-$; 
possibly $l_{\alpha}(x) = l_{\beta}(x)$. 
Similarly, let $r_{\alpha}(x)$, $r_{\beta}(x)$ be the tangent directions in $r(x)$ closest to $\alpha^+$, respectively $\beta^-$, along the arc $\alpha^+\beta^-$.
By definition, the angle between $l_{\alpha}(x)$ and $r_{\alpha}(x)$ towards $\alpha_+$ is at most $\pi$, 
as is the angle between $l_{\beta}(x)$ and $r_{\beta}(x)$ towards $\beta_+$. 
Because $S$ is orientable, for $z \in e \setminus \{x\}$ both $l(z)$ and $r(z)$ are inside precisely one of the above two angles, 
hence $z \not\in Q_y^{-1}$ and the claim follows.

\bs

The metric we considered is, by Theorem \ref{generic_Cx}, such that for any $x \in C_3^{cp}(y)$ and 
any two segments $\gamma, \gamma'$ joining $y$ to $x$, the angle of $\gamma, \gamma'$ at $y$ satisfies $\angle \gamma \gamma' \neq \pi$. 
Therefore, for any two cycles $C_i$ with $w(C_i) \neq 0$, the points $Q_y^{-1} \cap C_i$ are different and thus card$Q_y^{-1}$ is odd.
\end{proof}


\section{Torus case}
\label{exp}

In this section we show that the statement of Theorem \ref{odd_1} is sharp, in the sense pointed out by Theorem \ref{even}.

We will use the following result of A. D. Weinstein (Proposition C in \cite{W}).

\begin{lm}
\label{Wein}
Let $M$ be a $d$-dimensional Riemannian manifold and $D$ a $d$-disc embedded in $M$.
There exists a new metric on $M$ agreeing with the original metric on a neighborhood 
of $M \setminus ({\rm interior \; of \;} D)$ such that, for some point $p$ in $D$, 
the exponential mapping at $p$ is a diffeomorphism 
of the unit disc about the origin in the tangent space at $p$ to $M$, onto $D$.
\end{lm}

\begin{thm}
\label{even}
For any point $y$ on the torus $T$ there exist sets of metrics ${\cal E}^y_i$ on $T$, $i=1,2,3,4,5$, such that
${\rm card}Q_y^{-1}=i$ with respect to any metric ${\mathfrak g} \in {\cal E}^y_i$ and, moreover,
${\rm int}{\cal E}^y_j \neq \emptyset$ for $j=1,3,5$, while ${\cal E}^y_k$ contains continuous families of metrics for $k=2,4$.
\end{thm}

\begin{proof}
We indicate next a construction to get card$Q_y^{-1}=2$, but it can be easily adapted to obtain the conclusion.
(The stability of the respective constructions under small perturbations, for $j=1,3,5$, provides the non-empty interior.)

\medskip

Consider, in the hyperbolic plane ${\rm I\! H}$ of constant curvature $-1$, a circle $C$ of centre $y$ and radius $r$,
with $r>0$ a parameter to be chosen later. Figure \ref{fig0} illustrates in the plane our contruction.

\begin{figure*}
\label{fig0}
\centering
 \includegraphics[width=0.7\textwidth]{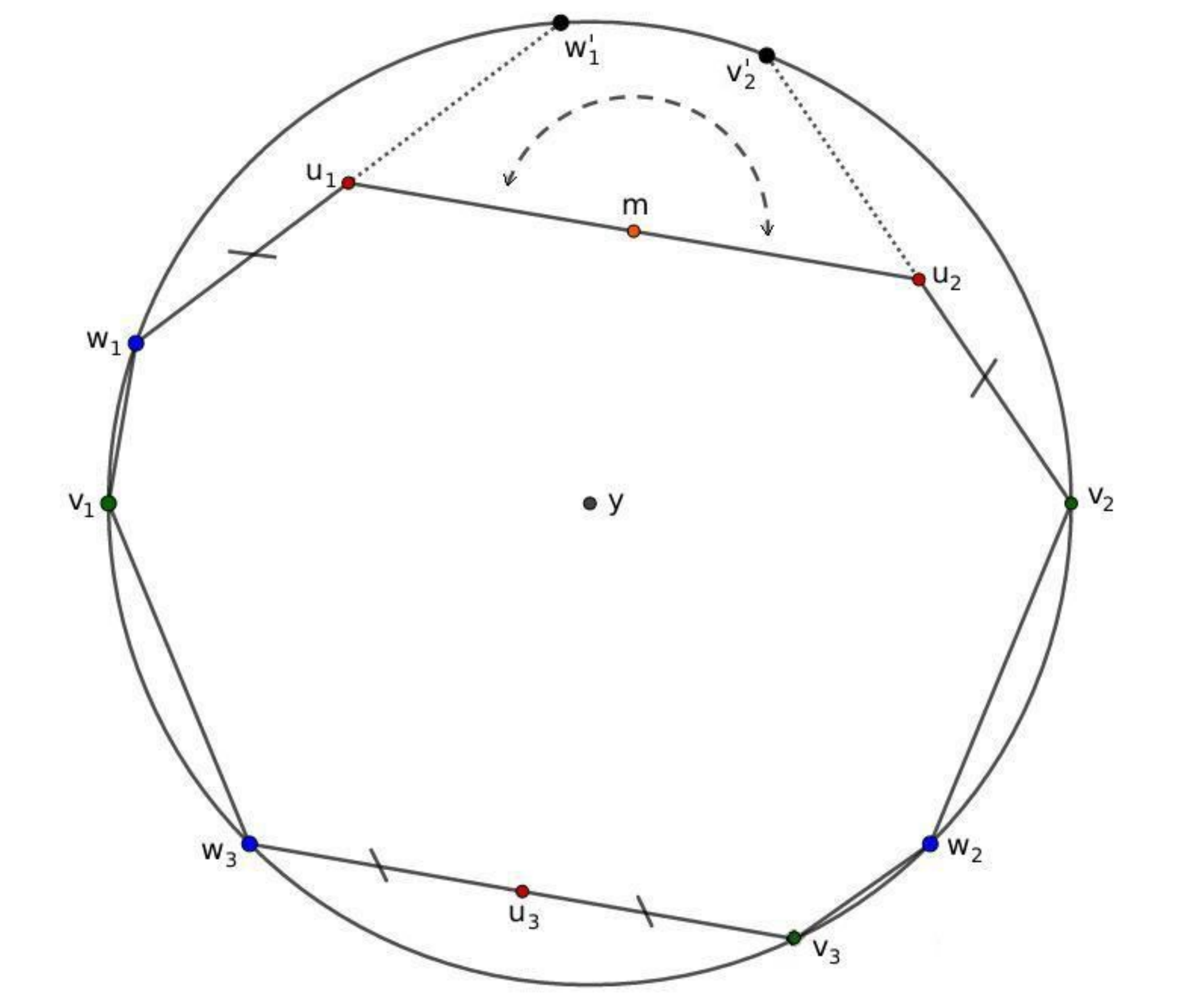}
\caption{Construction for a point $y$ on a torus, with card$Q_y^{-1}=2$.}
\end{figure*}

Consider points $v_1$ and $v_2$ diametrically opposite on $C$.
On one of the half-circles bounded by $v_1$ and $v_2$ consider points $w_3, u_3, w_2$ such that $v_2, w_2, v_3, w_3, v_1$ are in circular order
and $\lambda(v_1w_3)=\lambda(v_2w_2)>\lambda(v_3w_3)$.
On the other half-circle consider points $w_1, w_1', v_2'$ such that $\lambda(v_1w_1)=\lambda(v_3w_2)$ and 
$\lambda(w_1w_1')=\lambda(v_3w_3)=\lambda(v_2v_2')$. 
Of course, we may choose $w_2$ such that $v_1, w_1, w_1', v_2', v_2$ are in circular order.
Let $u_1$ be the mid-point of $w_1w_1'$, $u_2$ the mid-point of $v_2v_2'$, $m$ the mid-point of $u_1u_2$, and $u_3$ the mid-point of $v_3w_3$.

We may choose $r$ such that the total angle $\theta_v$ at $v$ verifies $\theta_v:=\angle w_1 v_1 w_3 + \angle u_2 v_2 w_2 + \angle w_2 v_3 w_3 =2\pi$.

Cut the polygon $v_1w_1u_1mu_2v_2w_3v_3u_3w_3v_1$ out from ${\rm I\! H}$ and naturally identify (glue along) the edges $mu_1$ and $mu_2$.
Further naturally identify the edges in the following pairs: $v_1w_1$ and $v_2w_2$, $v_2w_2$ and $v_1w_3$, $w_1u_1$ and $w_3u_3$, and $v_2u_2$ and $v_3u_3$.

Denote by $v$ the common image of $v_1,v_2,v_3$, by $w$ the common image of $w_1,w_2,w_3$,
by $u$ the common image of $u_1,u_2,u_3$, by $y$ the image of $y$, and by $m$ the image of $m$, via the above glueing procedure.

The resulting closed surface is a torus $T''$ with conical singularities at the points $m$ (where $\theta_m =\pi$),
$u$ (where $\theta_u =\angle w_1 u_1 m + \angle m u_2 v_2 + \pi > 2\pi$), and
$w_1$ (where $\theta_{w_1} =\angle v_1 w_1 u_1 + \angle v_2 w_2 v_3 + \angle v_3 w_3 v_1 > 2\pi$).

Smoothen first $T''$ locally around $m$ and $w$ to obtain a surface $T'$ with unique singularity at $u$.
Of course, small changes around those points do not affect the segments from $y$ to $v$ or $u$.
Moreover, because the directions of the segments from $w$ to $y$ were all included in an open half-circle of $T_y$, 
this property will remain true for all points in $T' \setminus T''$.

Next we show how to smoothen $T'$ around $u$.
Consider a metrical $\e$-neighbourhood $U_\e$ of $u$ on $T'$, of boundary length $l=l(\e, \theta_u) = \lambda(\partial U_\e)$.
Consider some $\alpha <-1$ such that, on the hyperbolic plane of constant curvature $\alpha$, 
the geodesic ball $D$ of radius $\e$ has boundary length precisely $l$.
Cut $U_\e$ off $T'$ and replace it by $D$. Also denote by $u$ the center of $D$ after the replacement.
By Lemma \ref{Wein}, there exists a torus $T$ whose metric outside a neighbourhood of $D$ coincides with the metric on $T'$,
and such that the directions of the segments from $u$ to $p=y$ remain the same as those on $T'$.
Therefore, on the obtained Riemannian surface $T$ we have $Q_y^{-1}=\{v,u\}$.

Of course, continuous changes of the positions of $w_1,w_2,w_3$ yield continuous families of metrics the with the desired property.
\end{proof}


\section{An upper bound for card$Q_y^{-1}$}
\label{upb}

\begin{thm}
\label{g}
Let $S$ be an orientable (Riemannian) surface of genus $g$ and $y$ a point in $S$.

If $g=0$ then ${\rm card} Q_y^{-1} =1$, and if $g=1$ then ${\rm card}Q_y^{-1} \leq 5$.

If $g\geq 2$ then ${\rm card} Q_y^{-1} \leq 8g-5$.
\end{thm}

\begin{proof}
The proof consists of two steps. First we prove directly that ${\rm card} Q_y^{-1} \leq 8g-4$, and 
afterward we invoke Theorem \ref{odd_1} to decrease that upper bound by $1$.

\medskip

{\it Step 1.}
In the virtue of Lemma \ref{sep}, we may consider only points $y \in S$ with $Q_y^{-1} \subset C^{cp}(y)$.

As in the proof of Theorem \ref{con}, we consider $C^{cp}(y)=(V,E)$ as a graph, 
with vertex set $V=C_3^{cp}(y)$ and edge set $E$ the set of components of $C^{cp}(y) \setminus V$. 
By Theorem \ref{con}, each edge of $C^{cp}(y)$ contains at most one interior point $x \in Q_y^{-1}$,
so $\sharp_y^2 \leq 6g-3$, $\sharp_y^3 \leq 4g-2$, and $\sharp_y^2 + \sharp_y^3 \leq 10g-5$.

Notice that this upper bound on card $Q_y^{-1}$ is imposed by the topology of $S$. 
We refine it next by local geometrical considerations.

\medskip

For the graph $C^{cp}(y)=(V,E)$, call an edge {\it white} if it intersects $Q_y^{-1}$, and {\it black} if it doesn't. 
A vertex is {\it white} if it belongs to $Q_y^{-1}$, and {\it black} otherwise.
A $Y$ is the subgraph of $C^{cp}(y)$ formed by a vertex $x$ of degree three and three edges issuing at $x$.

\medskip

Assume first that $C^{cp}(y)$ is a cubic graph.

We claim that, if there exists a white $Y$ in $C^{cp}(y)$, then no other edge is white.
To see this, assume the edges $e_{kl}$, $e_{km}$ and $e_{kn}$ are white and share a common extremity, say $v_k$.
Then the images on $T_y$ of the vertices incident to these edges respect the circular order
$v_l, v_k, v_m$, $v_n, v_k, v_l$, $v_m, v_k, v_n$. 
Since the images of each edge in the white $Y$ contain opposite points with respect to the centre of $T_y$, there is no place for other white edges.

Thus, if card$E=3$ then $S$ has genus $1$ (because it is orientable) and we get the upper bound ${\rm card} Q_y^{-1} \leq {\rm card} V + {\rm card} E =5$.
This is sharp, as one can easily see for a flat torus whose fundamental domain is a parallelogram.

If $g>1$ then, by our claim, at least one third of the edges are black.
Assuming all vertices are white, we obtain ${\rm card} Q_y^{-1} \leq 8g-4$.

\medskip

So we have obtained an upper bound ${\rm card}Q_y^{-1} \leq B_3(g)=8g-4$ if $C^{cp}(y)$ is a cubic graph.
We treat now the general case, in order to obtain an upper bound ${\rm card}Q_y^{-1} \leq B(g)$ 
with no restriction on the degree of vertices in $V$.

Slightly modify the metric ${\mathfrak g}$ of $S$ around the vertices of $C^{cp}(y)$ of degree larger than three to obtain a new metric ${\mathfrak g}'$ on $S$ close to ${\mathfrak g}$, with the following properties: every vertex in $C^{cp}(y)({\mathfrak g}')$ has degree three, and 
every white edge of $C^{cp}(y)({\mathfrak g})$ is still white in $C^{cp}(y)({\mathfrak g}')$.
This is possible by small perturbations of ${\mathfrak g}$ around (some of) the vertices 
$x \in C^{cp}(y)({\mathfrak g})$ with deg$x>3$ (see Theorem \ref{generic_Cx}).
Notice that, for ${\mathfrak g}'$ close enough to ${\mathfrak g}$, there cannot be more white edges in $C^{cp}(y)({\mathfrak g})$ than in $C^{cp}(y)({\mathfrak g}')$. As for the vertices, two or more black neighbours in $C^{cp}(y)({\mathfrak g}')$ may correspond to a white vertex of degree larger than $3$ in $C^{cp}(y)({\mathfrak g})$, which reduces to repaint in white at most half of the non isolated black vertices of $C^{cp}(y)({\mathfrak g})$. Thus, we get
\begin{equation}
\label{b}
B(g) \leq B_3(g) + \frac12 {\rm card} \left( V \setminus \left( Q_y^{-1} \cup
\left\{v \in V : v \; {\rm is} \; {\rm black} \; {\rm and} \; {\rm isolated} \right\} \right) \right) .
\end{equation}
Since our upper bound $B_3(g)$ assumes all vertices are white, the inequality (\ref{b}) gives
$$B(g) \leq B_3(g)=8g-4$$
and the proof of Step 1 is complete.

\medskip

{\it Step 2.}
Consider now metrics on $S$ as in Theorem \ref{odd_1}, hence the upper bound in this case is odd, namely
$B_3(g)^{{\rm odd}}=8g-5$.

The proof of Theorem \ref{odd_1} also shows that, if its cardinality is not odd, $Q_y^{-1}$ contains ``double'' points; i.e., points corresponding to several cycles. This, of course, implies that in case ${\rm card}Q_y^{-1}$ is even, $Q_y^{-1}$ doesn't have maximum number of elements. Therefore,
$B(g) \leq B_3(g)\leq B_3(g)^{{\rm odd}}=8g-5$ and the proof is complete.
\end{proof}


\section{Applications}

With the special case of mutually critical points deals \cite{Z2}. A
yet more particular case is that of pairs of points at distance equal
to the largest distance on $S$,
$$d(S)={\rm max}_{x,y\in S}\rho(x,y).$$
For any $x\in S$, we call $\rho_x ^{-1}(d(S))$ the {\it diametrally
opposite set} of $x$, if it is not void. 
In this case, the point $x$ itself is called {\it diametral}; of course, not every point is necessarily diametral.

Notice that  any diametrally opposite set verifies
$\rho_x^{-1}(d(S))\subset Q_x\cap Q_x^{-1}$ for any $x\in S$.

If $S$ is homeomorphic to S$^2$, every diametrally opposite set
contains a single point, by Theorem 1 in \cite{BIVZ} (see also \cite{vz2}). In the
standard projective plane, every point is diametral and any
diametrally opposite set is a circle. If $S$ is orientable, every
diametrally opposite set is finite, by Theorem \ref{con} or Theorem \ref{g}.

In analogy with the characterization provided by Theorem 2 in \cite{BIVZ} (given here as Lemma \ref{0}), 
we may think of a similar one imposing cardinality 1 for all diametrally opposite sets. But this condition is weaker.
Although surfaces homeomorphic to S$^2$ satisfy, by Theorem 1 in \cite{BIVZ}, the imposed condition, there are further examples of surfaces verifying it: 
any flat torus with a rectangular fundamental domain has only single-point diametrally opposite sets.

In any flat torus without a rectangular fundamental domain, the diametrally 
opposite set of every point $x$ has exactly $2$ points.

A direct consequence of Theorem \ref{g} is the following.

\begin{co}
\label{di}
For any point $y$ on an orientable surface of genus $g\ge 2$, the set $F^{-1}_y$ has at most $8g-5$ points. 
Hence any diametrally opposite set has at most $8g-5$ points.
\end{co}

Concerning the tightness of Theorem \ref{g} (and Corollary \ref{di}) we obtain the following.

\begin{thm}
\label{diam}
There exist orientable (Riemannian) surfaces $\tilde{T}_g$ of genus $g$ with diametrally opposite sets consisting of $4g+1$ points, 
where $2g+1$ points are of type $2$ and $2g$ points are of type $3$.
\end{thm}

\begin{proof}
For the case of surfaces homeomorphic to $S^2$, see Theorem 1 in \cite{BIVZ}.

Take now a flat torus with a parallelogram, union of two equilateral triangles with a common edge, as fundamental domain.

In such a torus, for any point $y$, $C(y)$ is a $\Theta$-shape graph. Cut
 along $C(y)$ and unfold to obtain a regular hexagon $v_1v_2v_1v_2v_1v_2$.
(If $y$ is taken to be the identified vertices of the parallelogram,
then $v_1$ and $v_2$ are the centres of the two triangles.)
Replace small discs of radius $\e$ about the midpoints $m_1, m_2, m_3$
 of the three distinct edges of the hexagon and about the centre (also
 denoted by) $y$ of the hexagon, by congruent bumps, all bounded by circles of
 length $2\pi\e$.
The bumps have centres $\tilde{m}_i$ and $\tilde{y}$ at distance
$\frac{1}{\sqrt{3}}-\frac{1} {2}+\e$ from the respective boundaries.
In this way we obtain a torus $\tilde{T}_1$, on which
$$\rho(\tilde{y}, v_1)=\rho(\tilde{y},v_2)=\rho(\tilde{y},\tilde{m}_1)
=\rho(\tilde{y},\tilde{m}_2)=\rho(\tilde{y},\tilde{m}_3)=
\frac{1}{\sqrt{3}}=d(\tilde{T}_1).$$
Thus, $\{v_1, v_2,\tilde{m}_1, \tilde{m}_2, \tilde{m}_3\}$ is a
diametrally opposite set of $\tilde{y}$.

\medskip

Next we define inductively surfaces $T_g$ for all $g \geq 2$, with the 
following properties.

The domain $D_g = T_g \setminus C(y)$ is a regular $6(2g-1)$-gon of
 centre $y$ in the hyperbolic plane of constant curvature $-1$, with the
 property that all its angles are $2 \pi/3$.

The cut locus of $y$ in $T_g$ is a --horizontally sitting-- tower-shape graph
 with $2g+1$ levels. Each level-edge provides a point in $Q_y^{-1}$ of type
 $2$, where from $\sharp_y^2 = 2g+1$, and each vertex of even level is of
 type $3$, so $\sharp_y^3 = 2g$.
Figure \ref{fig1} shows the case $g=2$, as well as where to attach a handle to
 $T_2$ in order to obtain the order of vertices on $D_3$.

\begin{figure*}
\label{fig1}
\centering
 \includegraphics[width=0.8\textwidth]{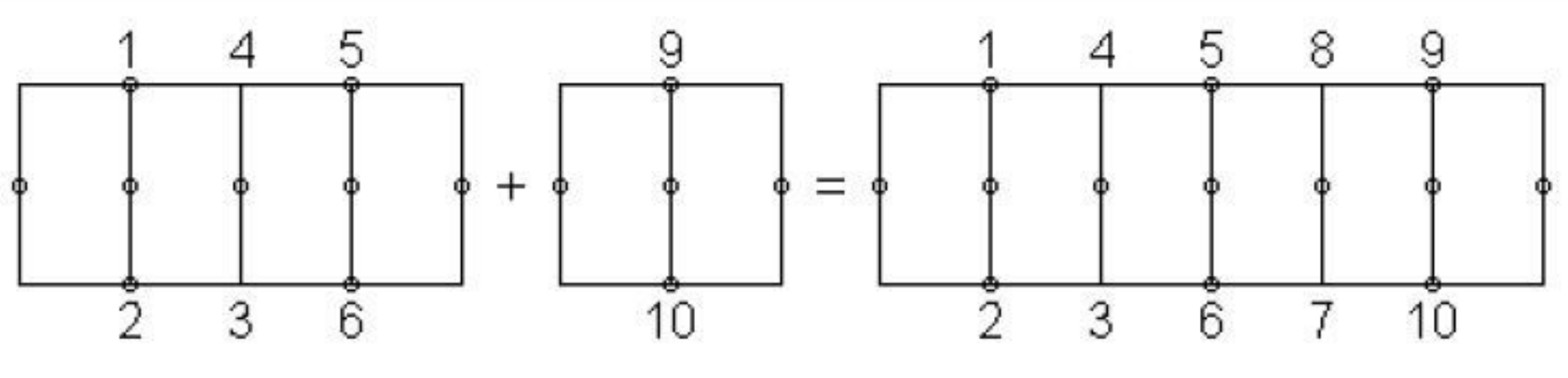}
\caption{Inductive construction for tower graphs. Glueing a surface $T_2$ of genus $2$ to a torus $T_1$, to obtain the surface $T_3$ of genus $3$: the right-most edge of $C(y)$ on $T_2$ is identified to the left-most edge of $C(y)$ on $T_1$, to get $C(y)$ on $T_3$. The points in $Q_y^{-1}$ are marked by small circles.}
\end{figure*}

To see that we can realize the tower-shape graphs as cut loci, it needs to 
specify how to identify (i.e., the order of) vertices and edges on $D_g$.

The domain $D_2=T_2 \setminus C(y)$ is a regular $18$-gon whose
 vertices, given in circular order, are $1$, $2$, $3$, $4$, $5$, $6$, $5$, $4$,
 $1$, $2$, $1$, $4$, $3$, $6$, $5$, $6$, $3$, $2$.
The edges, following the above order of vertices, are
$a$, $b$, $c$, $d$, $e$, $f$, $d$, $g$, $h$, $a$, $g$, $c$, $i$, $e$,
 $f$, $i$, $b$, $h$, see Figure \ref{fig2}.
Clearly, only the points $1$, $2$, $5$, $6$ are of type $3$,
and only the edges $a$, $c$, $e$, $f$, $h$ of $C(y)$ contain each
a point of type $2$, hence $\sharp_y^2 =5$, $\sharp_y^3 =4$, and
${\rm card}Q_y^{-1}=9$.

\begin{figure*}
\label{fig2}
\centering
 \includegraphics[width=\textwidth]{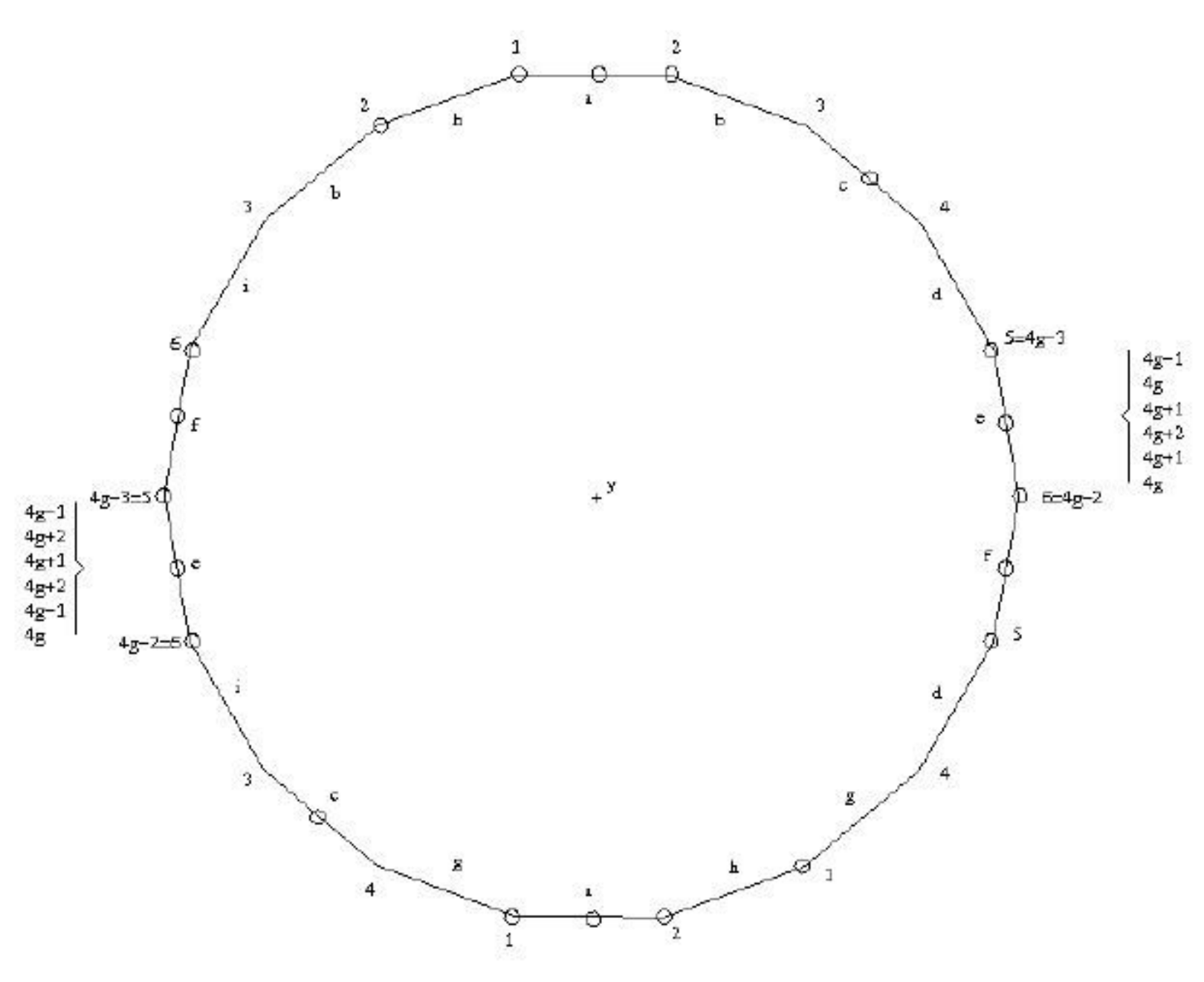}
\caption{Domain $D_2$. The points in $Q_y^{-1}$ are marked by small circles.}
\end{figure*}

Assume we have $T_g$ and $y \in T_g$ as above. Choose the right-most
 edge of $C(y)$, say $e$, and attach along it a handle. This reduces to locate
 the two images of $e$ on bd$D_g$ and to insert between their extremities 
(labeled $4g-2$ and $4g-3$) the points $4g-1, 4g+2, 4g+1, 4g+2, 4g-1,
 4g$, and respectively $4g-1, 4g, 4g+1, 4g+2, 4g+1, 4g$, see again Figure \ref{fig2}. Label the vertices of $D_{g+1}$ with the new obtained order. Identify the edges in the
 obvious way to obtain $T_{g+1}$, and notice that $\sharp_y^2 = 2g+1$, $\sharp_y^3 = 2g$.

Finally, replace (as in the case $g=1$) small disks about the midpoints
 of the distinct edges of bd$D_{g+1}$ and about the centre $y$ of
 $D_{g+1}$, by congruent bumps of centres $\tilde{x}_i$, $\tilde{y}$ in order to
 obtain $\rho(\tilde{y},\tilde{x}_i)=d(\tilde{T}_{g+1})$, where
 $\tilde{T}_{g+1}$ is the constructed surface.
\end{proof}


\section{Open questions}

Our approach leaves open several problems, among which we
state in the following only three that we find particularly interesting.

\begin{enumerate}

\item
The number of points with respect to which a point $y$ on a flat torus
 is critical, does not depend on $y$. This and Theorem 1 in \cite{BIVZ} lead us to
 the following question.

Find all surfaces $S$ with the property that card$Q_y^{-1}$ does not depend on $y \in S$.

\item
For the first step in the proof of Theorem \ref{g}, we considered points $x \in Q^{-1}_y$ which are vertices of $C^{cp}(y)$, 
and white subgraphs $Y$ of $C^{cp}(y)$ centered at $x$.
In other words, if we endow the graph $C^{cp}(y)$ with the discrete natural metric, we considered $1$-neighbouthoods of the points in 
$Q^{-1}_y \cap C_3^{cp}(y)$. 
Would the use of $k$-neighbourhoods, with $k \geq 2$, improve the upper bound?

\item
Every orientable surface of genus $g>0$ possesses points $x,y$ such that $y \in Q_x$ and there are two segments from $y$ to $x$ with opposite tangent directions at $y$ 
(see the proof of Theorem 2 in \cite{BIVZ}).

Is the same true for all surfaces homeomorphic to the sphere?
Or, at least, is it true for densely many surfaces homeomorphic to the sphere?

For a similar -- still open -- problem concerning convex surfaces, see \cite{z-ep}.
\end{enumerate}


\bigskip

\noindent {\bf Acknowledgement} 

Jin-ichi Itoh was partially supported by Grant-in-Aid for Scientific Research(C)(17K05222) JSPS.

C. V\^\i lcu  acknowledges partial financial support from the grant of the Roumanian Ministry of Research and Innovation, CNCS-UEFISCDI, project no. PN-III-P4-ID-PCE-2016-0019.

T. Zamfirescu thankfully acknowledges financial support by the High-end Foreign Experts Recruitment Program of People's Republic of China.
He is also indebted to the International Network GDRI Eco – Math for its support. 



Jin-ichi Itoh

\noindent{\footnotesize School of Education, Sugiyama Jogakuen University
\newline 17-3 Hoshigaoka-motomachi, Chikusa-ku, Nagoya, 464-8662 Japan}

{\small \hfill j-itoh@sugiyama-u.ac.jp}

\medskip

Costin V\^\i lcu

\noindent{\footnotesize {\sl Simion Stoilow} Institute of Mathematics of the Roumanian Academy
\newline P.O. Box 1-764, 014700 Bucharest, Roumania}

{\small \hfill Costin.Vilcu@imar.ro}

\medskip

Tudor Zamfirescu

\noindent{\footnotesize Fachbereich Mathematik, Universit\"at Dortmund
\newline
44221 Dortmund, Germany
\newline and
\newline  {\sl Simion Stoilow} Institute of Mathematics of the Roumanian Academy
\newline Bucharest, Roumania
\newline
and\newline College of Mathematics and Information Science,\newline Hebei
Normal University,
\newline050024 Shijiazhuang, P.R. China.}

{\small \hfill tudor.zamfirescu@mathematik.tu-dortmund.de}


\small

\begin{thebibliography}{99}

\bibitem{BIVZ}
I. B\'ar\'any, J. Itoh, C. V\^\i lcu and T. Zamfirescu,
{\it Every point is critical}, Adv. Math. {\bf 235} (2013), 390-397

\bibitem{Berger} M. Berger,
{\sl A Panoramic View of Riemannian Geometry},
Springer-Verlag, Berlin-New York, 2003

\bibitem{Bu}  M. A. Buchner, {\it The structure of the cut locus in dimension less than or equal to six}, Compositio Math. {\bf 37} (1978), 103-119

\bibitem{BGP} Y. Burago, M. Gromov and G. Perelman,
{\it A. D. Alexandrov spaces with curvature bounded below},
Russian Math. Surveys {\bf 47} (1992), 1-58

\bibitem{cfg}  H. T. Croft, K. J. Falconer and R. K. Guy,
{\sl Unsolved Problems in Geometry}, Springer-Verlag, New York, 1991

\bibitem{gg}  F. Galaz-Garc\'\i a, L. Guijarro,
{\it Every point in a Riemannian manifold is critical}, 
Calc. Var. Partial Differ. Equ. {\bf 54} (2015), 2079-2084

\bibitem{gro}  K. Grove, 
{\it Critical point theory for distance functions},
Amer. Math. Soc. Proc. of Symposia in Pure Mathematics, vol. {\bf 54} (1993), 357-385

\bibitem{gp}  K. Grove and P. Petersen,
{\it A radius sphere theorem},
Inventiones Math. {\bf 112} (1993), 577-583

\bibitem{i-z} J. Itoh and T. Zamfirescu, 
{\it On the length of the cut locus on surfaces}, 
Rend. Circ. Mat. Palermo, Serie II, Suppl. 70 (2002) 53-58

\bibitem{ko}  S. Kobayashi, {\it On conjugate and cut loci}, Global
differential geometry, MAA Stud. Math. {\bf 27} (1989), 140-169

\bibitem{p} H. Poincar\'e, {\it Sur les lignes g\'eod\'esiques des surfaces
convexes}, Trans. Amer. Math. Soc. {\bf 6} (1905), 237-274

\bibitem{ro2}  J. Rouyer,
{\it On farthest points on Riemannian manifolds homeomorphic to the $2$-dimensional sphere},
Rev. Roum. Math. Pures Appl. {\bf 48} (2003), 95-103

\bibitem{ro4}  J. Rouyer,
{\it On antipodes on a manifold endowed with a generic Riemannian metric},
Pacific J. Math. {\bf 212} (2003), 187-200

\bibitem{sa}  T. Sakai, {\sl Riemannian Geometry}, Translation of
Mathematical Monographs 149, Amer. Math. Soc. 1996

\bibitem{ST}  K. Shiohama and M. Tanaka,
{\it Cut loci and distance spheres on Alexandrov surfaces},
Actes de la Table Ronde de G\'eom\'etrie Diff\'erentielle (Luminy, 1992),
S\'em. Congr., vol. 1, Soc. Math. France, 1996, 531-559

\bibitem{v2}  C. V\^\i lcu,
{\it Properties of the farthest point mapping on convex surfaces},
Rev. Roumaine Math. Pures Appl. {\bf 51} (2006), 125-134

\bibitem{v3}  C. V\^\i lcu,
{\it Common maxima of distance functions on orientable Alexandrov surfaces},
J. Math. Soc. Japan {\bf 60} (2008), 51-64

\bibitem{vz2} C. V\^\i lcu and T. Zamfirescu,
{\it Multiple farthest points on Alexandrov surfaces},
Adv. Geom. {\bf 7} (2007), 83-100

\bibitem{W}  A. D. Weinstein, {\it The cut locus and conjugate locus of a riemannian manifold}, Ann. Math. {\bf 87} (1968), 29-41 

\bibitem{ZI} T. Zamfirescu,
{\it Many Endpoints and Few Interior Points of Geodesics},
Invent. Math. {\bf 69} (1982), 253-257

\bibitem{ZP} T. Zamfirescu,
{\it On the cut locus in Alexandrov spaces and applications to convex surfaces},
Pacific J. Math. {\bf 217} (2004) 375-386

\bibitem{Z2} T. Zamfirescu,
{\it Antipodal trees and mutually critical points on surfaces},
Adv. Geom. {\bf 7} (2007), 385-390

\bibitem{z-ep}  T. Zamfirescu,
{\it Extreme points of the distance function on convex surfaces},
Trans. Amer. Math. Soc. {\bf 350} (1998), 1395-1406

\end{thebibliography}
\end{document}